\documentclass[a4paper, 11pt, reqno]{amsart} 
\usepackage[dvipsnames]{xcolor}
\usepackage[english]{babel}               

\usepackage[T1]{fontenc} 

  \usepackage{libertinus}
  \usepackage[amsthm]{libertinust1math}

\usepackage[utf8]{inputenc} 
 
\usepackage{amsthm}
\theoremstyle{plain}
\newtheorem{theorem}{Theorem}[section]
\newtheorem{proposition}[theorem]{Proposition}
\newtheorem*{proposition*}{Proposition}
\newtheorem*{theorem*}{Theorem}
\newtheorem{lemma}[theorem]{Lemma}
\newtheorem*{lemma*}{Lemma}
\newtheorem{conjecture}[theorem]{Conjecture}

\theoremstyle{definition}

\newtheorem{remark}[theorem]{Remark}
\usepackage{mathrsfs}      

\usepackage{tikz} 
\usepackage{enumerate}    
\usepackage{mathtools} 
\usepackage{textcomp}  
\mathtoolsset{showonlyrefs} 
\numberwithin{equation}{section}  

 \usepackage[margin=1.3in]{geometry}     
\usepackage[babel]{microtype} 

\usepackage[symbol]{footmisc}
 \usepackage{comment}      
 \usepackage{appendix}

\usepackage{graphicx}  
\usepackage{float}                      

\usepackage[
breaklinks   = true,        
pagebackref  = true,     
]{hyperref}
\hypersetup{
	colorlinks=true,
	pdfpagemode=UseNone,
    citecolor=Green,
    linkcolor=OrangeRed,
    urlcolor=black,
	pdfstartview=FitW
}

\def\N{\mathbb{N}}

\def\R{\mathbb{R}}
\def\E{\mathbb{E}}
\def\P{\mathbb{P}} 

\def\infV{\underline{V}}

\newcommand{\ind}[1]{\mathbf{1}_{\left\{ #1 \right\}}}    

\newcommand*{\dif}{\ensuremath{\mathop{}\!\mathrm{d}}}

 \renewcommand{\bar}[1]{\mkern 1mu\overline{\mkern-1mu #1\mkern-1mu}\mkern 1mu}
\renewcommand{\hat}[1]{\widehat{#1}}

\title[Exponential moments of truncated BRW martingales]{Exponential moments of truncated branching random walk martingales}

\author[H. Ma]{Heng Ma} 

\address[Heng Ma]
{Peking University, School of Mathematical Sciences, China
}
\email{hengmamath at gmail dot com}
\urladdr{\url{https://hengmamath.github.io}}

\author[P. Maillard]{Pascal Maillard} 

\address[Pascal Maillard]{Institut de Mathématiques de Toulouse, CNRS UMR 5219, Université de Toulouse, 118 route de Narbonne, 31062 Toulouse Cedex 9, France and Institut Universitaire de France.}

\email{Pascal.Maillard at math dot univ-toulouse dot fr}
\urladdr{\url{https://www.math.univ-toulouse.fr/~pmaillar/}} 


\date{\today}

\begin{document}
 
\begin{abstract}
  For a branching random walk that drifts to infinity, consider its Malthusian martingale, i.e.~the additive martingale with parameter $\theta$ being the smallest root of the characteristic equation. When particles are killed below the origin, we show that the limit of this martingale admits an exponential tail, contrary to the case without killing, where the tail is polynomial. In the critical case, where the characteristic equation has a single root, the same holds for the (truncated) derivative martingale, as we show. This study is motivated by recent work on first passage percolation on Erd\H{o}s--R\'{e}nyi graphs. 
\end{abstract}

\maketitle


\section{Main Results}
 
Consider a branching random walk (BRW) generated by a point process $\Xi$ on $\mathbb{R}$. The process starts with a single particle (the root $\rho$) at the origin at generation zero. At each generation $n \geq 1$, every particle from generation $n-1$ dies and is replaced independently by a random number of offspring whose displacements from the parent follow the distribution $\Xi$. 
The resulting genealogical structure is a Bienaymé--Galton--Watson tree $\mathcal{T}$, where each individual may have an infinite number of offspring.
For each particle $u \in \mathcal{T}$, let $V(u)$ denote its position, $|u|$ its generation, and $[\rho, u]$ the set of ancestors of $u$, including $u$ itself. Set
\begin{equation}
    \label{eq:Biggins-transform}
    \Phi (\theta) :=  \ln \mathbb{E} \left[ \int_{\mathbb{R}} e^{-\theta x} \,  \Xi(\dif x)  \right] =\ln \E \left[ \sum_{|u|=1} e^{- \theta V (u)}  \right] \in (-\infty,\infty].
\end{equation}
Throughout this paper, we assume 
\begin{equation}\label{eq-Assumption}
\Phi(0)  \in (0,\infty] \ \text{ and } \  \Phi(1) = 0.
\end{equation}
Furthermore, for simplicity, we assume that $\Phi$ is finite in a neighborhood of $1$.
In particular, these conditions imply that $\inf_{|u|=n} V(u) \to \infty$ almost surely, see e.g. \cite[Theorem 3]{Biggins98}.

Define
  \begin{equation}
    W_n := \sum_{|u|=n} e^{-V(u)}.
  \end{equation}
This forms a non-negative martingale known as the additive or Biggins’ martingale \cite{Biggins77}. If $\Phi '(1) = \E[ V(u) e^{-V(u)} ] < 0$, then $1$ is the smallest root of the \emph{characteristic equation} $\Phi(\theta) = 0$, by (strict) convexity of $\Phi$. The martingale $W_n$ is then sometimes called the \emph{Malthusian martingale} of the branching random walk. Under an $L \ln L$ condition, Biggins \cite{Biggins77} showed that the limit $W_\infty := \lim_{n \to \infty} W_n$ of the Malthusian martingale is non-degenerate.
In the so-called boundary case (terminology according to Biggins and Kyprianou \cite{BK05}) where $\Phi'(1) = 0$, the additive martingale vanishes almost surely, and one must instead consider the derivative martingale 
\begin{equation}
    D_{n} := \sum_{|u|=n} V(u) e^{-V(u)}  ,
\end{equation}
which converges almost surely to a positive limit $D_\infty$ on the survival set (see \cite{BK04} and \cite{Chen15})  under mild moment assumptions. 
The existence and positivity of $D_\infty$ is itself non‑trivial,  as $(D_{n})$ is signed martingale with $\E[D_{n}]=0$ and  $\sup_{n \geq 1} \E[ |D_{n}|]=\infty$.

In this article, we study the truncated additive martingale and the truncated derivative martingale, defined for $n\geq 0$ and $x \geq 0$ by 
\begin{align}
\label{eq:def_truncated_martingale}
 W_{n}^{(x)} &:= \sum_{|u|=n} e^{- ( x+V(u) ) } \ind{ x+ \infV(u) \geq 0 },\\
 \label{eq:def_truncated_derivative_martingale}
  D_{n}^{(x)} &:= \sum_{|u|=n} R(x+ V(u)) e^{- ( x+V(u) ) } \ind{ x+ \infV(u) \geq 0 }, 
\end{align}
where $  \infV(u) := \min \left\{  V(v):  v \in [\rho, u] \right\}$  and $R (\cdot ) \geq 0$ is the renewal function of the associated random walk (see \S \ref{sec-notation} for the precise definition).  Definitions~\eqref{eq:def_truncated_martingale} and \eqref{eq:def_truncated_derivative_martingale} can be interpreted as the effect of killing particles upon hitting the negative half-line.
For each $x\ge0$, the sequences $(W_n^{(x)})_{n\ge 0}$ and $(D_n^{(x)})_{n\ge 0}$ form a non-negative supermartingale and a martingale, respectively. Consequently, they converge almost surely to random limits $W_\infty^{(x)}$ and $D_\infty^{(x)}$:
\begin{equation}
  W_{n}^{(x)} \xrightarrow[n \to \infty]{a.s.} W_{\infty}^{(x)} \ , \   D_{n}^{(x)} \xrightarrow[n \to \infty]{a.s.} D_{\infty}^{(x)} .
\end{equation}
Moreover since 
  $\inf_{|u|=n}V(u)\to\infty$ a.s., we have 
$e^{x} W_{\infty}^{(x)} \xrightarrow[x \to \infty]{a.s.}  W_{\infty} $ and $e^{x} D_{n}^{(x)} \xrightarrow[n \to \infty]{a.s.} D^{(x)}   $.

Note that $D_{n}^{(x)}$ is well-known as a key tool to study the (untruncated) derivative martingale, see e.g. \cite[\S 5.2]{Zhan15}. Indeed, the existence of $D_\infty$ follows from the facts that $D_n^{(x)} \to D_\infty^{(x)}$ in $L^1$, $\inf_{|u|=n} V(u) \to \infty$ almost surely, and $R(y) \sim c y$ as $y \to \infty$ for some $c > 0$. This naturally raises the question of whether the truncated martingales $(D_n^{(x)})$ possess stronger properties than uniform integrability.  

Our first theorem concerns the tail of the truncated additive martingale:

\begin{theorem}[Tail of truncated additive martingale]
\label{thm-tail-W-x} 
Suppose $\Phi'(1)<0$ and set 
\begin{equation}
  \kappa := \inf\bigl\{\theta>1:\Phi(\theta)\ge0\bigr\}\in(1,\infty].
\end{equation} 
Let $\gamma\in (1,\kappa)$. If $W_1 = \sum_{|u|=1} e^{-V(u)}$  admits a finite exponential moment, then there are constants $C,c>0$ such that
\begin{equation}\label{tail-W-x}
  \P(W^{(x)}_{n} > y )  \leq  C  e^{ -\gamma x} e^{- c y}    
\end{equation}  
for all $x>0$, $y\ge1$, and $n\in\mathbb{N}\cup\{\infty\}$. 
\end{theorem}

Our second theorem concerns the tail probability of $D_{\infty}^{(x)}$. It yields that if both $W_{1}$ and the ``total variation'' of $D_{1}$ have exponential moment, then so does $D_{\infty}^{(x)}$.  
 
\begin{theorem}[Tail of truncated derivative martingale]\label{thm-tail-D-x}
Suppose $\Phi '(1) =0$. Define 
\begin{equation}
  X= \sum_{|u|=1} [1+ (V(u))_+] e^{- V(u)} . 
\end{equation}  
If the random variable $X$ admits a finite exponential moment, then  there exists constants $C,c>0$   such that 
\begin{equation}\label{tail-D-x}
    \P(D^{(x)}_{n} > y )  \leq  C \,  R(x)e^{-x} \, e^{-c y} 
\end{equation}   
for all $ x >0 $, $y \ge 1$  and $ n \in \mathbb{N} \cup \{ \infty \}$. 
\end{theorem}

\subsection{Background and related work}
Tail probabilities for both additive and derivative martingales in branching random walks have been extensively studied \cite{Guivarc'h90, Liu00, Buraczewski09, Madaule16b, BIM21,CDM24}. 
Regarding the additive martingale, Liu \cite{Liu00} proved that if  $\Phi(\kappa) = 0$ (with $\kappa\in (1,\infty)$) and certain moment conditions hold then 
 \begin{equation}
  \P( W_{\infty} > y) \sim \frac{c}{ y^{\kappa}}  \text{ as } y \to \infty.
 \end{equation}
Moreover, under mild moment conditions, Buraczewski, Iksanov, and Mallein \cite{BIM21} showed that the derivative martingale limit \(D_\infty\) exhibits Cauchy-type tails: 
 \begin{equation}
  \P( D_{\infty} > y) \sim  \frac{1}{y} \text{ and  } \, \E  \left[ D_{\infty} \ind{ D_{\infty} \leq y}  \right] = \ln y + c + o(1) \text{ as } y \to \infty.
 \end{equation} 
 Although truncation is the standard technique to establish convergence to a non-degenerate limit, existing literature has only shown uniform integrability of the truncated derivative martingale. 
 A natural question is then: what is the tail behavior of this truncated limit? Our Theorem \ref{thm-tail-D-x} answers this by demonstrating that its tail actually decays exponentially fast. This is maybe a surprise at first glance.   
   
  In the case where $\Xi$ is supported on $\mathbb{R}_+$ where $\kappa =\infty$,  R\"{o}sler \cite{Rosler95} showed that a finite exponential moment  of $W_{1}$ implies a finite exponential moment of $W_{\infty}$. It is natural to ask whether, when negative displacements occur, the truncation (or killing at negative half-line) can restore an exponential tail. We answer this question in the affirmative.

Our initial motivation comes from first-passage percolation on Erd\H{o}s--R\'{e}nyi graphs \cite{DalySchulteShneerFirstPassage2023}. 
In the case of non-negative weights, the authors analyze the extreme value statistics of shortest paths, using moment calculations. 
A crucial step is to show that the additive martingale in a certain branching random walk with non-negative jumps (Crump-Mode-Jagers process) has exponential tail, which they prove using R\"{o}sler's technique \cite{Rosler95}. 
If one allows for negative weights, then the local weak limit of the graph becomes a branching random walk with positive or negative jumps, in which case the additive martingale has a polynomial tail, so that the method of moments cannot be used. 
However, the results of the current article indicate that one can salvage this by a suitable truncation. This is left for future work.

In \cite{BerestyckiBrunetHarrisMilosBranchingBrownian2017}, the authors study the number of particles absorbed at the origin in branching Brownian motion in the case where the minimum has non-negative speed. They also show that the number of particles has an exponential tail, both in the case of positive speed and of zero speed. Their methods are analytic and rely on a detailed study of an associated differential equation, which does not adapt to the branching random walk. On the other hand, they obtain indeed a true asymptotic equivalent for the tail.
We expect that  the tail of $W^{(0)}_{\infty}$ should be related to the tail of the number of absorbed particles, since for either of these quantities, the strategy to make them large should be to let a few number of particles stay close to the origin for a long time. However, we leave a precise study of the large deviation event for future study (see also Section~\ref{sec:outlook}).

  On the methodological side, it is interesting to see that  R\"{o}sler's arguments can be extended to the subcritical case but seem to fail in the critical case (boundary case), whereas spine decompositions handle both settings (see Section~\ref{sec:proof_ideas} below). It is also maybe the first example of the use of spine techniques for showing that additive functionals have exponential tails. This could be of potential use in other contexts.
  
\subsection{Proof ideas}
\label{sec:proof_ideas}

We provide two proofs for Theorem~\ref{thm-tail-W-x}. The first one is an adaptation of R\"{o}sler's \cite{Rosler95} proof in the case of non-negative increments, where no truncation is necessary. R\"{o}sler's idea was to obtain a bound on the moment generating function of $W_n$ by induction on $n$. He was able to obtain such a bound of the form
\[
\E[e^{\theta W_n}] \le \exp(\theta + K \theta^2),\quad 0\le \theta \le \bar\theta,
\]
for some $\bar\theta > 0$ and $K<\infty$ independent of $n$. In our case, we get an extra dependence on the starting point $x$, but we can still show that one can obtain a similar bound (under a slightly stronger assumption than the one stated in Theorem~\ref{thm-tail-W-x}): for any $\rho\in (1,\kappa\wedge 2)$, there is $\bar\theta>0$ and $K<\infty$, such that
\[
\E[e^{\theta W_n^{(x)}}] \le \exp(\theta e^{-x} + K \theta^{\rho}e^{-\rho x}),\quad 0\le \theta\le \bar\theta,\ x\ge 0.
\]
This bound directly shows that $W_n^{(x)}$ admits a finite exponential moment (independent of $n$), but does not allow to obtain the dependence on $x$ of the tail. We obtain this from a tail bound for the total minimum of the (unkilled) branching random walk, using essentially that the branching random walk has to send a particle close to the origin in order for $W^{(x)}_n$ to be large. We refer to Section~\ref{sec:additive_martingales} for details. 

For the proof of Theorem~\ref{thm-tail-D-x}, we were not able to extend the above proof (note that we have $\kappa = 1$ in this case). We therefore establish a completely different proof, using a spine decomposition techniques. Again, we aim to bound the moment generating function
\[
f_n(\theta,x) = \E[e^{\theta D_n^{(x)}}].
\]
We observe that its derivative in $\theta$ has the expression
\[
f_n'(\theta,x) = \E[D_n^{(x)}e^{\theta D_n^{(x)}}].
\]
We now apply a well-known \emph{spine decomposition} in order to express the above quantity as
\[
f_n'(\theta,x) = R(x)e^{-x} \widehat{\E}_x^*[e^{\theta D_n^{(x)}}],
\]
with $\widehat{\E}_x^*$ the expectation with respect to a certain branching random walk with a distinguished spine following a centered random walk with finite variance, conditioned to stay non-negative. The quantity $D_n^{(x)}$ being an additive functional of the branching random walk, we can decompose it along the particles on the spine. Showing finiteness of the exponential moment of $D_n^{(x)}$ under this law then, roughly speaking, amounts to studying finiteness of a certain randomized additive functional of the trajectory of the spine, see Lemma~\ref{lem-CondRW}. This functional this related to so-called  \emph{perpetuities}, see e.g. the comprehensive monograph by Iksanov~\cite{IksanovRenewalTheory2016}, or \cite{AIR09}. However, as opposed to classical perpetuities, the underlying random walk here is not a random walk with a drift, but a centered random walk conditioned to stay non-negative. We prove in Lemma~\ref{lem-CondRW} that this particular perpetuity has an exponential tail. Its proof works by decomposing the conditioned random walk into excursions and using the fact that the number of excursions from a given set can be bounded by a geometric random variable with a certain parameter. We refer to Section~\ref{sec:proof_derivative_martingale} for details.

The proof of Theorem~\ref{thm-tail-D-x} can be generalized to provide an independent proof of Theorem~\ref{thm-tail-W-x}, under the assumption stated in the theorem. We provide a sketch of the argument in Remark~\ref{rmk-spine-argu-to-W}.

\subsection{Outlook and open problems}
\label{sec:outlook} 

It would be very interesting to obtain more precise tail asymptotics for the truncated martingales, such as an asymptotic equivalent for the tails $\P(W_\infty^{(x)} > y)$ and $\P(D_\infty^{(x)} > y)$ as $x$ and $y$ go to infinity. Concerning the dependence on $x$, we conjecture that it is fairly simple and should basically amount to the probability of the existence of a particle getting close to the origin starting from $x$. In other words, we expect the following to hold under mild assumptions:
\begin{conjecture}
For some $K\in (0,\infty)$,
\[
\P(W_\infty^{(x)} > y) \sim Ke^{-\kappa x} \P(W_\infty^{(0)} > y),\quad x,y\to\infty,
\]
and similarly, for some $K\in (0,\infty)$,
\[
\P(D_\infty^{(x)} > y) \sim Ke^{-x} \P(D_\infty^{(0)} > y).
\]
\end{conjecture}
On the other hand, obtaining the precise tail of $\P(W_\infty^{(0)}> y)$ and $\P(D_\infty^{(0)} > y)$ might be tricky. This would require a precise understanding of the best strategy for leading to a large deviation event. Naively, one could expect the best strategy to be to ensure that a single particle, a ``spine'', stays close to the origin for a sufficiently long time. However, it is not clear that this is the best strategy, and we indeed believe that there might be situations in which is it beneficial to have several of such particles, leading to a ``skeleton'' rather than a single spine. Investigating this further would be quite interesting.

Relatedly, it would be interesting to study the tail of other quantities for the truncated branching random walk, notably, the number of absorbed particles as mentioned above.

Finally, it would be interesting to study more generally the ``perpetuities'' arising in this study, see Section~\ref{sec:proof_ideas} and Lemma~\ref{lem-CondRW}. For example, for a given function $f:\R_+\to\R_+$, such that $\sum_{n=0}^\infty f(S_n^+)$ is finite almost surely, what is its tail? In particular, how fast must $f$ decay so that the tail is exponential?

\subsection{Overview of paper}

In Section~\ref{sec:additive_martingales}, we study the tail of the truncated additive martingales, providing the proof of Theorem~\ref{thm-tail-W-x}. Section~\ref{sec:derivative_martingale} is devoted to the study of the tail of the truncated derivative martingale. In Subsection~\ref{sec-notation}, we recall basic notions regarding killed random walks and spine decompositions. The proof of Theorem~\ref{thm-tail-D-x} appears in Section~\ref{sec:proof_derivative_martingale}.

\section{Truncated Additive Martingales}
\label{sec:additive_martingales} 

In this section, we prove Theorem~\ref{thm-tail-W-x} by adapting  R\"{o}sler's proof \cite[Theorem 6]{Rosler95} for branching random walks with non-negative increments. Throughout the section, we assume that the assumptions of Theorem~\ref{thm-tail-W-x} are verified. We assume furthermore that there exists $\gamma \in(1,\kappa)$ such that the random variable $\sum_{|u|=1} e ^{-\gamma V(u)}$ admits a finite exponential moment. This assumption is necessary for the proof from this section, however, as explained in Remark~\ref{rmk-spine-argu-to-W} below, an alternative way to prove Theorem~\ref{thm-tail-W-x} is to adapt the proof of Theorem~\ref{thm-tail-D-x} in the next section. Hence, Theorem~\ref{thm-tail-W-x} still holds as stated.


For each $\theta \geq 0$, we define the log-moment generating function of $W^{(x)}_{n}$ as 
\begin{equation}
   \psi_{n}( \theta,x ) : =\ln  \E \left[  e^{ \theta W_{n}^{(x)} } \right] \in (0,\infty ] \, , \ \forall \,  n \in \mathbb{N} \cup \{\infty \}.
\end{equation}
In particular $\psi_0(\theta,x)= \theta e^{-x}$ for all $\theta \geq 0$ and $x \geq 0$.

\begin{lemma}\label{lem-moments-W-x}
  Under the above assumptions, fix  an arbitrary $\rho\in(1,  \gamma \wedge 2]$. There exist constants $K,\bar{\theta}>0$ such that for every   $n\in\mathbb{N}\cup\{\infty\}$, 
  \begin{equation}\label{eq-exp-moment-addi}
   \psi_{n}(\theta, x) \leq \theta e^{-x} + K \theta^{\rho}e^{-\rho x}  \, , \, \forall \, \theta\in[0,\bar{\theta}],  x \ge 0 . 
  \end{equation} 

\end{lemma}


Theorem~\ref{thm-tail-W-x} follows directly from the lemma above, whose proof we defer to the end of this section.

\begin{proof}[Proof of Theorem \ref{thm-tail-W-x} assuming Lemma \ref{lem-moments-W-x}]
  Fix $\rho \in (1,\ \gamma \wedge 2]$. 
  It follows from the Markov inequality and Lemma \ref{lem-moments-W-x}   that, for any   $ y \geq 1$, the following inequality holds: 
  \begin{equation}\label{eq-exp-decay-W-x}
  \P( W^{(x)}_{n} > y)   \leq  e^{-\bar\theta y} \E [ e^{\bar\theta  W^{(x)}_{n}}]  \leq e^{ \bar{\theta} e^{-x} + K \bar{\theta}^{\rho}e^{-\rho x}} e^{ - \bar{\theta} y} \leq  C e^{     - \bar{\theta }y}
  \end{equation}
  for all $n \geq 0$ and $x \ge 0$, where we choose constant $C \ge e^{\bar\theta + K \bar\theta^{\rho}}$. Define 
  \begin{equation}
    c  :=   \frac{   \bar{\theta}  {(\kappa-\gamma)}} {2 \kappa} \in (0,\bar{\theta}).
  \end{equation} 
  Whenever $(\bar{\theta} - c) y \geq \gamma x$, it follows that $e^{-\bar{\theta} y} \leq e^{-\gamma x - c y}$.  Together with \eqref{eq-exp-decay-W-x}, we conclude that 
  \begin{equation}
    \P( W^{(x)}_{n} > y)  \le  C e^{- \gamma x} e^{-c y}
  \end{equation} 
  as desired.  

  It remains to treat the case  $1 \leq  y \leq  \frac{\gamma}{\bar{\theta}-c} x$. Let $\mathbf{I}_{n} := \inf_{|u| \leq n} V(u)$. For any $0\leq z < x$,  observe that on the event $\{  x+ \mathbf{I}_{n} \geq z \}$  we have 
  \begin{equation}
     W^{(x)}_{n}     = e^{-z}\sum_{|u|=n} e^{- ( x-z+V(u) ) } \ind{ x-z + \infV(u) \geq 0 } = e^{-z}   W^{(x-z)}_{n}  .
  \end{equation}  
  It follows from the  inequality \eqref{eq-exp-decay-W-x} that 
  \begin{equation}
     \P(  W^{(x)}_{n}  \geq y , x+ \mathbf{I}_{n} \geq z ) \leq   \P(  W^{(x-z)}_{n}  \geq e^{z} y   ) \leq C e^{-\bar{\theta} e^{z} y }.
  \end{equation}
  Moreover, by Lemma 3.1 in \cite{CDM24}, we know
  \begin{equation}
    \P(x + \mathbf{I}_n \leq z) \leq e^{-\kappa(x - z)} .
  \end{equation} 
 Combining these two inequalities and choosing $z= \ln^2(1+x)\le x$ yields
  \begin{equation}
     \P(  W^{(x)}_{n}  \geq y ) \leq  e^{-\kappa(x-\ln^{2}(1+x))} + C e^{-\bar{\theta} e^{\ln^{2}(1+x)} y }.
  \end{equation} 
for all $1 \leq   y \leq \frac{  \gamma} { \bar{\theta}-c  }x$.  
Observe that, since $c < \bar{\theta}\frac{\kappa-\gamma}{\kappa}$, we have $\gamma\left[1+\frac{c}{\bar{\theta}-c}\right] < \kappa$. Hence, for sufficiently large $x_o$ we have 
  \begin{equation}
   e^{- \gamma x - c y} \geq e^{- \gamma [ 1+ \frac{c}{\bar{\theta}-c}] x } \geq 
   e^{-\kappa(x-\ln^{2}(1+x))}  \  , \  \forall \, x \ge x_o,  1 \le y \leq \frac{  \gamma} { \bar{\theta}-c  }x . 
  \end{equation}   
  Further, for sufficiently large $  x_o$, we also have $\bar{\theta} e^{\ln^{2}(1+x)} \geq \gamma x + c $ for $x \ge x_o$  which implies that  
  \begin{equation}
   e^{- \gamma x - c y} \geq  e^{- [\gamma x +c] y} \geq  e^{-\bar{\theta} e^{\ln^{2}(1+x)} y }   \  , \  \forall \, x \ge x_o,  1 \le y \leq \frac{  \gamma} { \bar{\theta}-c  }x . 
  \end{equation}  
  Thus, for $x\geq x_o$ and $1\le y\le \frac{\gamma}{\bar{\theta}-c}x$, we conclude 
  \begin{equation}
     \P(  W^{(x)}_{n}  \geq y ) \leq (1+C) e^{- \gamma x - c y}   .
  \end{equation}   
  For the case $x \le x_o$ and $y\le \frac{\gamma}{\bar{\theta}-c}x_0$, we simply enlarge the constant $C$, thus completing the proof. 
  \end{proof}

\begin{proof}[Proof of Lemma \ref{lem-moments-W-x}]
   Assume that \eqref{eq-exp-moment-addi} holds for $n-1$. To extend it to generation $n$, we use the usual decomposition of $W^{(x)}_{n}$ at generation $1$.   For any individual $v \in \mathcal{T}$, let $\mathcal{T}^{(v)}$ denote the subtree 
   rooted at $v$, 
   and $\mathcal{T}^{(v)}_{k}$ it's $k$-th level. Define   $V^{(v)}(u):= V(u)-V(v)$ for all $u \in \mathcal{T}^{(v)}$. Then it follows that 
   \begin{align}
       W^{(x)}_{n} &= \sum_{|v|=1} \ind{V(v)+x \geq 0} \underbrace{\sum_{u \in \mathcal{T}^{(v)}_{n-1}} e^{- x+ V(v)+ V^{(v)}(u) } \ind{ x+ V(v) + \infV^{(v)}(u) \geq 0 }}_{ W^{(x+V(v),\, v)}_{n-1}} \\
       &=  \sum_{|v|=1} \ind{V(v)+x \geq 0}  W^{(x+V(v),\, v)}_{n-1} . 
   \end{align} 
   The branching property yields that conditionally on $\mathcal{F}_{1}:=\sigma( V(u) : |u|=1 )$,  the processes $\left\{ ( V^{(v)}(u) : u \in \mathcal{T}^{(v)} ) :  |v|=1 \right\} $ are independent BRWs and hence $ \{ W^{(x+V(v), v)}_{n-1} : |v|=1 \}$ are independent random variables. Consequently we obtain that 
   \begin{equation}
       e^{\psi_{n}(\theta,x)}   =   \E \left(  \prod_{|v|=1,V(v)+x \geq 0 }\E\left[  e^{\theta W^{(x+V(v), v)}_{n-1}  } \mid  \mathcal{F}_{1}\right]\right)   = \E \left(  \prod_{|v|=1,V(v)+x \geq 0 }  e^{\psi_{n-1} (\theta, x+V(v))}\right) .
   \end{equation}
   Applying the inductive bound $\psi_{n-1}(\theta,x+V(v))\le \theta e^{-(x+V(v))}+K\,\theta^\rho e^{-\rho(x+V(v))}$  gives 
   \begin{align}
       e^{\psi_{n}(\theta,x)}  & \leq  \E \left[ \exp \left( \theta e^{-x} \sum_{|v|=1} e^{-V(v)} \ind{V(v) \geq -x}  + K \theta^{\rho} e^{-\rho x}  \sum_{|v|=1} e^{-\rho V(v)} \ind{V(v) \geq -x}  \right)\right] \\
       &\leq \exp \left(  \theta e^{-x} + K \theta^{\rho}e^{-\rho x} \right) G_{K}(\theta e^{-x}),
   \end{align}
   where   the function $G_K$ is defined as 
   \begin{equation}
       G_{K}(\lambda) :=   \E \left[ \exp \left(\lambda [ \sum_{|v|=1} e^{-V(v)}  -1 ] + K \lambda^{\rho}  [\sum_{|v|=1} e^{-\rho V(v)} -1 ] \right)\right].
   \end{equation}
  By our assumption and the fact $ \sum_{|v|=1} e^{-\rho V(v)} \leq \sum_{|v|=1} e^{- \gamma V(v)} + \sum_{|v|=1} e^{-  V(v)}$, there exists  $\delta>0$ such that   $ G_{K}(\lambda)$ is well defined for $\lambda \in [0,\delta]$. Moreover the dominated convergence theorem yields that $ G_{K}(\lambda)$ is indeed infinitely differentiable on $(0,\delta)$.  

  It suffices to show that there exists positive constants $K$ and $\bar{\theta}$  such that
  \begin{equation}
    G_{K}(\lambda) \leq 1 \, \text{ for all } \, \lambda \in [0, \bar{\theta}].
  \end{equation}   
 Clearly $G_{K}(0)=1$. For brevity, set $W_1(\rho)=\sum_{|v|=1} e^{-\rho V(v)}$.  
 Taking  the derivative, we obtain that  
   \begin{equation}
       G_{K}'(\lambda)=   \E \left[ \left(  [W_{1}-1] + K \rho \lambda^{\rho-1}  [ W_{1}(\rho) -1 ]  \right) \exp \left(\lambda [W_{1} -1 ] + K \lambda^{\rho}  [ W_{1}(\rho) -1 ] \right)\right]. 
   \end{equation}
   for  $\lambda \in (0,\delta)$ and $G_{K}'(0)=0$. Differentiating once more yields 
   \begin{align}
       G_{K}''(\lambda) &=   \E \left[ \left( K \rho (\rho-1) \lambda^{\rho-2}  [ W_{1}(\rho) -1 ]  \right) \exp \left(\lambda [W_{1} -1 ] + K \lambda^{\rho}  [ W_{1}(\rho) -1 ] \right)\right] \\
      &  +     \E \left[ \left(  [W_{1}-1] + K \rho \lambda^{\rho-1}  [ W_{1}(\rho) -1 ]  \right)^2 \exp \left(\lambda [W_{1} -1 ] + K \lambda^{\rho}  [ W_{1}(\rho) -1 ] \right)\right] 
   \end{align}
 for  $\lambda \in (0,\delta)$.
Since $\rho \in (1,2]$,   the continuity of $G''(\lambda)$  implies that  for  $K$ sufficiently  large and $\lambda$ sufficiently small, say $0<\lambda < \bar{\theta}$, $G_{K}''(\lambda) \leq 0$. Hence $G_{K}(\lambda) \leq 1$ for $0 \leq \lambda \leq \bar{\theta}$   and we conclude   that \eqref{eq-exp-moment-addi} holds for any integer $n \in \mathbb{N}$. 
 Finally, applying Fatou's lemma extends the bound to $\psi_\infty$, completing the proof.
\end{proof}


\section{Truncated Derivative Martingales}
\label{sec:derivative_martingale}
 
\subsection{Definitions and preliminaries}
 \label{sec-notation}
 We first introduce some definitions and notation. In this section, we work under the assumption $\Phi '(1) = \E[ V(u) e^{-V(u)} ] =0$.

 \subsubsection*{Associated random walk.}
  Let $\{(S_{n})_{n \ge 0}, \mathbf{P}\}$ be the random walk with step distribution given by
\begin{equation}
  \mathbf{E} [ f(S_{1}) ] :=  \E \left[ \sum_{|u|=1} f(V(u)) e^{-V(u) }   \right]  \ , \ \forall f \in \mathrm{C}_{\mathrm{b}}(\mathbb{R}).
\end{equation}
In particular, $\mathbf{E}[S_1]= -\Phi'(1)=0$; and 
$\mathbf{E}[ e^{\theta S_{1}}] = e^{\Phi(1-\theta)}$, which is finite for small $|\theta|$ by hypothesis.   
For any $x\in \R$,  let $\mathbf{P}_{x}$  denote the law of the walk starting from $S_0=x$ and write $\mathbf{P}=\mathbf{P}_0$ for simplicity.

\subsubsection*{Associated renewal function.}
Let $R:[0, \infty) \rightarrow(0, \infty)$ be the renewal function of the strictly descending ladder heights of the random walk $(S_n)_{n\ge0}$, which can be expressed as follows by the duality lemma \cite[Section~XII.2]{Feller1971}:
\begin{equation*}
R(u):=\mathbf{E}  \left[ \sum_{j=0}^{\tau^{+}-1} \mathbf{1}_{\left\{S_j \geq-u\right\}}  \right]  \text{ for } u \geq 0, 
\end{equation*} 
where $\tau^{+}:=\inf \left\{n \geq 1: S_{n} \geq 0\right\}$ represents the first passage time of $[0,\infty)$. 
In particular, $R(0)=1$. Since  $\mathbf{E}[S_1^2]<\infty$,  it is well-known that (see Rogozin~\cite{RogozinDistributionFirst1964} or Doney~\cite{Doney80}) if in addition $S_1$ is non-arithmetic, then 
\begin{equation}\label{eq-renewal-thm}
  \lim_{u \to \infty}  R(u+1)-R(u) =   c_{\mathrm{ren}} \in (0,\infty) .
\end{equation}  
If $S_{1}$ is arithmetic with span $h>0$, then   
\begin{equation}
  \lim_{n \to \infty} R((n+1) h)-R(nh)  =    c_{\mathrm{ren}} \in (0,\infty) . 
\end{equation}
Moreover, the renewal function 
$R$ is harmonic for the random walk killed when entering $(-\infty,0)$ by Tanaka's theorem (see Tanaka~\cite{Tan1989} or Shi \cite[Lemma 4.7]{Zhan15}):
\begin{equation}
  \mathbf{E} \left[  R(x+S_1) \ind{x+S_1 \geq 0} \right] = R(x) \, , \, \forall \, x \ge 0.
\end{equation}

\subsubsection*{Random walk conditioned to be nonnegative.} 
Let $\{(S_{n}^+)_{n \ge 0}, \mathbf{P}_x \}$ denote the random walk $(S_{n})$ conditioned to be nonnegative in the sense of Doob's $h$-transform. Specifically, $(S_{n}^+)_{n \ge 0}$ is defined as a Markov chain with transition probabilities
\begin{equation}
  \mathbf{P} ( S^{+}_{n} \in \dif y \mid S^{+}_{n-1}= x ) = \frac{R(y)}{R(x)} \,  \mathbf{P}( S_{1}\in \dif y \mid S_{0}= x ) \text{ for all } n \ge 1, x , y \ge 0.
\end{equation}

\subsubsection*{Spine Decomposition.}  
Since $(D_n^{(x)})_{n \ge 0}$ is a nonnegative martingale with mean  $\E[D_n^{(x)}]  =R(x)e^{-x}$, by Kolmogorov's consistency theorem, there exists a probability measure $\hat{\P}_{x}$ such that  
\begin{equation}
 \dif \hat{\P}_{x}  |_{\mathcal{F}_n}  = \frac{D_n^{(x)}}{R(x)e^{-x}} \cdot \dif  \P |_{\mathcal{F}_n}  \, , \, \forall \, n \ge 0.
\end{equation}
To describe the law of the  process under the tilted measure $\hat{\P}_{x}$, one need to enlarge the probability space
and consider the BRW with a spine, defined as follows.

The system starts with one particle, denoted by $w_0$, at position $V(w_0)=0$. 
At each step $n$, particles of generation $n-1$ die, while reproducing independently as follows. Define for every $y\ge 0$ a point process $\hat{\Xi}_y$ with law given by
\begin{equation}
  \E  \left[ F ( \hat{\Xi}_{y}) \right]  =  \E \left[ F(\Xi) \int_{\mathbb{R}} \frac{R(y+z)}{R(y)} e^{- z }\ind{y+z \ge 0}  \, \Xi(\dif z )   \right] , 
\end{equation}
for every bounded measurable function $F$. The particle $w_{n-1}$ generates a point process distributed relative to its position  
as $\hat{\Xi}_{x+V\left(w_{n-1}\right)}$, whereas any particle $u$, with $|u|=n-1$ and $u \neq w_{n-1}$, generates a point process   distributed relative to its position as $\Xi$.
Then, the  particle $w_{n}$ is chosen among the children $v$ of $w_{n-1}$ with probability proportional to $R(x+V(v)) e^{-V(v)} \ind{ x+\underline{V}(v) \geq 0 }$. The line of descent $(w_n)_n \ge 0$ is referred to as the spine. We denote by $\hat{\P}^{*}_{x}$ the distribution of the BRW with a spine $\left\{  (V(u))_{ u \in \mathcal{T}}, (w_n)_{n \ge 0} \right\} $.

  \begin{proposition}[\cite{BK04}]\label{spine-decomposition}
    Fix $x \ge 0$. For every $n \ge 0$,  $ \hat{\P}^*_{x}  \big|_{ \mathcal{F}_{n}} =  \hat{\P}_{x}   \big|_{ \mathcal{F}_{n}}  $. Moreover, 
    \begin{enumerate}[(i)]
      \item For any $n$ and any vertex $u \in \mathcal{T}$ with $|u|=n$, we have 
\begin{equation}
  \hat{\P}^*_{x}  \left(   w_n=u \mid \mathcal{F}_n   \right)=\frac{R(x+V(u)) \mathrm{e}^{-[x+V(u)]} }{D_n^{(x)}} \ind{x+\underline{V}(u) \geq 0 }.
\end{equation}
 
\item The spine process $\{ (x+ V(w_n))_{n \ge 0}, \hat{\P}^{*}_{x}  \}$  has the same distribution as the random walk conditioned to be nonnegative  $\{ (S^+_{n})_{n \ge 0}, \mathbf{P}_{x}\}$. 
    \end{enumerate}
  \end{proposition} 

\subsection{Proof of Theorem~\ref{thm-tail-D-x}}
\label{sec:proof_derivative_martingale}
 
Similar to  the last section we define for each 
$n \in \N \cup \{ \infty \}$, 
\begin{equation}
   f_{n}(  \theta,x) := \E \left[  e^{ \theta D^{(x)}_{n}} \right] \, , \, \forall \, \theta \ge 0, x \ge 0.
\end{equation}
In particular $f_0( \theta,x)=  \theta R(x) e^{-x}$. 
Notice that if $f_{n}(\bar{\theta}  ,x)<\infty$ for some $\bar\theta>0$, then by Proposition \ref{spine-decomposition}, for all $\theta \in (0,\bar\theta)$    we have 
\begin{equation}\label{eq-f-prime-theta-x}
   f_{n}'( \theta,x) = \E \left[ D^{(x)}_{n}  e^{ \theta D^{(x)}_{n}} \right]  = R(x)e^{-x} \hat{\E}^{*}_{x}[ e^{ \theta  D^{(x)}_{n}}  ].    
\end{equation}

\begin{lemma}\label{lem-exp-mon-D-x}
Under the assumptions of Theorem~\ref{thm-tail-D-x}, there exists $\bar{\theta}>0$, such that for all  $n \in \mathbb{N} \cup \{\infty\}$ and $x\ge 0$,
\begin{equation}
\label{eq-exp-moment-2}
 \hat{\E}^{*}_{x}[ e^{\bar{\theta}  D^{(x)}_{n}}  ] \leq 2 . 
\end{equation}
  As a consequence, we have 
     \begin{equation}\label{eq-exp-moment-deri}
    \ln f_{n}(\theta, x) \leq  2 \theta R(x) e^{-x}  \text{ for any } \theta \in [0,\bar{\theta}], x \geq 0. 
     \end{equation}
  \end{lemma}

   \begin{proof}[Proof of theorem \ref{thm-tail-D-x} assuming Lemma \ref{lem-exp-mon-D-x}] Applying Markov's property, Proposition \ref{spine-decomposition} and  Lemma \ref{lem-exp-mon-D-x}, 
   for $y \geq 1$  we have 
     \begin{align}
         \P(D^{(x)}_{n} > y ) & \leq  \E[ D^{(x)}_{n}] \E \left[ \frac{D^{(x)}_{n}}{\E[ D^{(x)}_{n}]  } \ind{ D^{(x)}_{n} > y }    \right]   = R(x) e^{-x} \hat{\P}^{*}_{x}(D^{(x)}_{n} > y  ) \\
         & \leq  R(x) e^{-x} e^{-\bar{\theta} y} \hat{\E}^{*}_{x}[ e^{\bar{\theta}  D^{(x)}_{n}}  ]  \leq 2  R(x)e^{-x} e^{-\bar{\theta} y} .
     \end{align}
       This completes the proof.
    \end{proof}

 The remainder of this section is devoted to the proof of Lemma~\ref{lem-exp-mon-D-x}. Before proceeding, we state a lemma which shows that a certain functional associated with the random walk conditioned to stay nonnegative admits a finite exponential moment.

\begin{lemma}\label{lem-CondRW}
  Let $(S_{n}^+)$ denote the Doob $h$-transform of the  $(S_{n})$ and  $(Q_{n})$ be a random sequence such that for all non-negative, measurable function $f$,
  \begin{equation}\label{AssumptionConditionallLaw}
     \mathbf{E} \left[ f(  S^+_{n+1}, Q_{n+1} ) | (S^+_{j}, Q_{j}), 1 \leq j \leq n\right] =  \mathbf{E}_{x} [ f(S^{+}_1,Q_{1})]|_{x= S_{n}}.
  \end{equation} 
  If  there is $ \delta >0$ and $C_{\delta}>0$ such that 
  \begin{equation}\label{eq-as-Q-moment}
         \mathbf{E}_{x}[ e^{  \delta Q_{1} } ] \leq  e^{C_{\delta}}  \ , \ \forall \, x \geq 0 . 
  \end{equation}  
  Then there exists $ \epsilon >0$   such that 
  \begin{equation}
   \sup_{x \ge 0} \mathbf{E}_{x} \left[ \exp \left(   \epsilon  \sum_{k=0}^{\infty} R( S^{+}_{k} )e^{- S^{+} _{k} } Q_{k+1}\right) \right] \le 2. \end{equation}
  \end{lemma}
  \vspace{3pt}

   \begin{proof}[Proof of Lemma \ref{lem-exp-mon-D-x} assuming Lemma \ref{lem-CondRW}] 
    We proceed by induction. The case $n=0$ holds because $D_0^{(x)} = R(x)e^{-x}$ is bounded. Now let $n\ge 1$ and assume that \eqref{eq-exp-moment-2} holds for all $k \leq n-1$, for some $\bar{\theta}>0$. We now prove that \eqref{eq-exp-moment-2} holds for $n$, with $\bar{\theta}$ possibly replaced by $\bar\theta \wedge c_0$, for some constant $c_0>0$ independent of $n$ and $x$.
    
    For $k\le n$, denote by $\mathcal{B}(w_{k})$ the set of siblings of the spine particle $w_k$, i.e.~, those $u\in \mathcal{T}$ such that $|u|=k$, $u\ne w_k$ and the parent of $u$ is $w_{k-1}$. Furthermore, as in the proof of Lemma~\ref{lem-moments-W-x}, denote by $\mathcal{T}^{(u)}$ the subtree of $u$. 
    Under $\hat{\P}^{*}_{x}$,  we have almost surely
     \begin{equation}
       D^{(x)}_{{n}} = \sum_{k=1}^{n} \sum_{u \in \mathcal{B}(w_{k})} \ind{ x+ \infV(u) \geq 0} D^{(x+V(u),u)}_{n-k} +  R(x+V(w_{n}))e^{- (x+V(w_{n}))} ,
     \end{equation}
     where $D^{(y,u)}_{n}:=  \sum_{v \in \mathcal{T}^{(u)}_{n}} R(y+V^{(u)}(v)) e^{-[y+ V^{(u)}(v) ]} $ for $y \ge 0$ and $n \in \mathbb{N}$. 
     The branching property yields that conditionally on $\mathcal{G}^{*}_{n}:= \sigma( V(u) , u \in \mathcal{B}(w_{k}), 1 \leq k \leq n )$,  the processes $\{ ( V^{(u)}(v) : v \in \mathcal{T}^{(u)} ) :  u \in \cup_{k=1}^{n} \mathcal{B}(w_{k}) \} $ are independent BRWs and hence 
     $ \{ D^{(x+V(u), u)}_{n-|u|} :  u \in \cup_{k=1}^{n} \mathcal{B}(w_{k}) \}$ are conditionally independent.
 Thus we can rewrite $\hat{\E}^{*}_{x} [   e^{ \theta D^{(x)}_{n}}  ] $ as 
 \begin{align}
   & \hat{\E}^{*}_{x} \left[  \exp \left(  \theta   \sum_{k=1}^{n} \sum_{u \in \mathcal{B}(w_{k})} \ind{ x+ \infV(u) \geq 0} D^{(x+V(u),u)}_{n-k} +  \theta R(x+V(w_{n}))e^{- (x+V(w_{n}))}   \right) \right]   \\
  &=    \hat{\E}^{*}_{x}  \left[   \exp\left( \theta R(x+V(w_{n}))e^{- (x+V(w_{n}))} \right) \prod_{k=1}^{n} \prod_{ \substack{ u \in \mathcal{B}(w_{k}) \\ x+ \infV(u) \geq 0 }} f_{n-k} ( \theta, x+ V(u))   \right]. 
 \end{align} 
Define 
 \begin{equation}
  A_{n}  :=  \sum_{k=1}^{n} \sum_{u \in \mathcal{B}(w_{k})} R(x+V(u))  e^{-(x+V(u))} \ind{ x+ \infV(u) \geq 0} .
 \end{equation}
 Then by applying the induction hypothesis we obtain that  
 \begin{equation}
  \hat{\E}^{*}_{x}\left[   e^{ \theta D^{(x)}_{n}} \right]  \leq \hat{\E}^{*}_{x}  \left[   e^{ \theta \left(2  A_{n} +  R(x+V(w_{n}))e^{- (x+V(w_{n}))}\right) }  \right]   .
 \end{equation}
 Let us denote $\Delta V(u):= V(u)- V(\overleftarrow{u})$ for each $u \in \mathcal{T}\backslash \{ \rho \}$ where
 $\overleftarrow{u}$ denotes the parent of $u$. Then using the fact that there exists constant $C_{\mathrm{R}}>0 $ satisfying $R(x+y) \leq C_{\mathrm{R}} (R(x)+ y_+)$ for all $x \geq 0 $ and $x+y \ge 0$, we get that 
 \begin{align}
A_{n} & = \sum_{k=0}^{n-1}  e^{- (x+V(w_{k})) }   \sum_{u \in \mathcal{B}(w_{k+1})} R(x+V(w_{k})+\Delta V(u))  e^{- \Delta V(u) } \ind{x+V(w_{k})+ \Delta V(u) \geq 0}   \\
& \leq C_{\mathrm{R}}  \sum_{k=0}^{n-1} R( x+V(w_{k})) e^{- (x+V(w_{k})) } \underbrace{  \sum_{u \in \mathcal{B}(w_{k+1})} [ 1+(\Delta V(u))_+ ] e^{- \Delta V(u) }  }_{\eqqcolon\Delta_{k+1}} .
 \end{align} 
According to Proposition \ref{spine-decomposition}, the process $\{(x+V(w_n))_{n \ge 0}, \hat{\P}_x\}$ has the same law as $\{( S^+_{n})_{n \ge 0},\mathbf{P}_x \}$. To apply Lemma~\ref{lem-CondRW}, we first note that condition \eqref{AssumptionConditionallLaw} follows directly from  the branching property of the BRW with a spine.  It remains to verify condition \eqref{eq-as-Q-moment}. Notice that $\Delta_1 \le X$, with $X$ defined in the statement of Theorem~\ref{thm-tail-D-x}. Under the assumptions of that theorem, there exists $\delta>0$ such that $\E[Xe^{\delta X}]<\infty$. Then by Proposition \ref{spine-decomposition}, for all $x\geq 0$, we have 
\begin{align}
   \hat{\E}^{*}_{x} \left[  e^{\delta \Delta_1} \right] \le \hat{\E}^{*}_{x} \left[  e^{\delta X} \right] &=  \E  \left[  \sum_{|u|=1} \frac{R(x+ V(u))}{R(x)} e^{-V(u)} \ind{V(u)+x \geq 0} e^{\delta X} \right] \\
   & \leq C_{\mathrm{R}} \E  \left[  \sum_{|u|=1} [1+ (V(u))_+]   e^{-V(u)}   e^{\delta X} \right] = C_{\mathrm{R}} \E  \left[  X e^{\delta X} \right] < \infty,  \label{eq-moment-Delta-1}
\end{align} 
as required. 
Therefore by using Lemma \ref{lem-CondRW} with $Q_{k+1} = 1+2C_{\mathrm{R}}\Delta_{k+1}$, provided that $ \bar{\theta}$ is sufficiently small (independent of $n$ and $x$)
 we get 
   \begin{equation}
    \hat{\E}^{*}_{x}\left[   e^{ \bar\theta D^{(x)}_{n}} \right]  \leq \hat{\E}^{*}_{x}  \left[    e^{ \bar\theta 2 C_{\mathrm{R}} \sum_{k=0}^{n-1}R( x+V(w_{k}) )e^{- (x+V(w_{k})) } \Delta_{k+1} + \bar\theta R( x+V(w_{n}) )e^{- (x+V(w_{n})) } }      \right] \leq 2.
   \end{equation}
   This proves \eqref{eq-exp-moment-2}.
   
Finally, \eqref{eq-f-prime-theta-x} and \eqref{eq-exp-moment-2} together imply for every $\theta\le \bar\theta$,
   \begin{equation}
  f_{n}(\theta,x)  =f_{n}(0,x) + \int_{0}^{\theta} f'_{n}(\lambda,x) \dif \lambda   \leq  1+  2 \theta R(x)e^{-x} \leq  e^{  2 \theta R(x)e^{-x}},
   \end{equation}
which proves \eqref{eq-exp-moment-deri}.
\end{proof}

Now it remains to proof Lemma \ref{lem-CondRW}. We assume for convenience that the random walk $(S_n)_{n\ge 0}$ is non-arithmetic. The argument can be readily adapted to the arithmetic case, and we omit the details.

\begin{proof}[Proof of Lemma \ref{lem-CondRW}]
  By \eqref{eq-renewal-thm}, there  exists some constant $c_{\mathrm{R}}$ such that $R(y)\leq c_{\mathrm{R}} e^{y/2}$ for all $y \geq 0$.  Thus 
  \begin{align}
     \sum_{k=0}^{\infty} R( S^{+}_{k} )e^{- S^{+} _{k} } Q_{k+1}
     &\leq c_{\mathrm{R}}  \sum_{j=0}  e^{-j/2} \sum_{k \geq 0} \ind{S_k^+ \in [j,j+1)} Q_{k+1} .
  \end{align} 
  Define $c_0 \coloneqq \left( \sum_{j=0}^{\infty} e^{-j/6}\right)^{-1}  $. 
  We obtain that for large $M \geq 0$,
    \begin{align}
        & \mathbf{P}_{x}\left( \sum_{j=0}  e^{-j/2} \sum_{k \geq 0} \ind{S_k^+ \in [j,j+1)} Q_{k+1}  >  M \right) \\
          & =  \mathbf{P} _{x}\left(\sum_{j=0}^{\infty}  c_{\mathrm{R}} e^{-j/2}  \sum_{k \geq 0} \ind{S_k^+ \in [j,j+1)} Q_{k+1}  > c_0 \sum_{j=0}^{\infty}  e^{- j/6 } M  \right) \\
      & \leq   \sum_{j=0}^{\infty} \mathbf{P}_{x} \left(   \sum_{k \geq 0} \ind{S_k^+ \in [j,j+1)} Q_{k+1} > \frac{c_{\mathrm{R}}}{c_0} e^{j/3} M  \right). \label{eq-bound-1}
    \end{align} 

  In order to decouple the dependence between $(S_k^{+})_{k\ge0}$ and $(Q_k)_{k\ge 1}$, we fix $j \in \mathbb{N}$ and decompose the trajectory into excursions from the intervals $[j,j+1)$. Each excursion begins when the chain is in the interval $[j, j+1)$ and ends once the chain first reaches $[j+2, \infty)$. Explicitly, for every $j \geq 0$ and $m \geq 1$, we define $\tau^{\mathrm{sta}}_{m}(j)$ and $\tau^{\mathrm{end}}_{m}(j)$ as the starting and ending times of the $m$-th excursion of the Markov chain $(S^{+}_n)_{n \geq 1}$, by  $\tau^{\mathrm{end}}_{0}=0$, and for $m\ge 1$,
\begin{align}
    \tau^{\mathrm{sta}}_{m}(j)  &:= \inf \left\{ k \geq \tau^{\mathrm{end}}_{m-1} : S^{+}_{k} \in [j,j+1) \right\}, \\
    \tau^{\mathrm{end}}_{m}(j)  & := \inf \left\{ k \geq   \tau^{\mathrm{sta}}_{m} +1 : S^{+}_{k} \geq  j+2 \right\},
\end{align}
where $\inf \emptyset = + \infty$.  Let
\begin{equation}
     \zeta(j) := \sum_{m=1}^{\infty} \ind{    \tau^{\mathrm{sta}}_{m}(j) < \infty  } 
\end{equation}
 denote the number of such excursions and  for each $1 \leq m \leq \zeta(j) $, define
\begin{equation}
    \mathcal{Q}_{m}(j):=  \sum_{k=   \tau^{\mathrm{sta}}_{m}(j) }^{\tau^{\mathrm{end}}_{m}(j)-1  } \ind{S_k^+ \in [j,j+1)} \, Q_{k+1}   .
\end{equation}  
Then it follows that 
\begin{equation}\label{eq-bound-27}
  \sum_{k \geq 0} \ind{S_k^+ \in [j,j+1)} Q_{k+1} =  \sum_{m=1}^{\zeta(j)}   \mathcal{Q}_{m}(j) .
\end{equation}

 Fix $m \geq 2$.  
 Applying the strong Markov property of $(S^+_{n},Q_{n})_{n \ge 0}$ and assumption \eqref{AssumptionConditionallLaw},
  and by conditioning the process at times $\tau^{\mathrm{end}}_{1}(j)$ and  $\tau^{\mathrm{sta}}_{2}(j)$, we obtain that,  for any $\lambda \geq 0$ 
\begin{align}
  E_{\eqref{eq-exp-moments-m-excursion}}(\lambda,m;j) & :=  \sup_{x \ge 0} \mathbf{E}_{x} \left[  \exp\left(  \lambda   \sum_{k=1}^{\zeta(j)}   \mathcal{Q}_{k}(j)   \right) \ind{\zeta(j) = m} \right] \label{eq-exp-moments-m-excursion} \\ 
    & \leq  \sup_{y \in [j,j+1)} \mathbf{E}_{y}\left[  \exp   \left(  \lambda  \mathcal{Q}_{1}(j)   \right)  \right] \times \sup_{x \geq j+2} \mathbf{P}_{x}( \exists n \geq 0, S^{+}_{n} < j+1 )  \\
    & \qquad \times  \sup_{y \in [j,j+1)} \mathbf{E}_{y} \left[  \exp\left(   \sum_{k=1}^{\zeta(j)} \mathcal{Q}_{k}(j) \right) \ind{\zeta(j) = m-1} \right].
\end{align}
We claim that there are constants $c_o, C_o$ such that 
for any $0 \leq \lambda \leq c_o$, and $j \ge 0$
\begin{equation}\label{eq-exp-moments-per-excursion}
    F_{\eqref{eq-exp-moments-per-excursion}}(\lambda;j)  := \sup_{y \in [j,j+1)} \mathbf{E}_{y}\left[  \exp   \left(  \lambda  \mathcal{Q}_{1}(j) \right)  \right] \leq e^{C_o \lambda} 
\end{equation}  
Then  by iteration we get that for any $0 \leq \lambda < c_o$,
\begin{equation}
    E_{\eqref{eq-exp-moments-m-excursion}}(\lambda,m;j) \leq  e^{C_o \lambda m} \left[  \sup_{x \geq j+2} \mathbf{P}_{x}( \exists n \geq 0, S^{+}_{n} < j+1 )   \right]^{m-1} . 
\end{equation}
We further note that for all $0\le y \leq x$, we have \cite[Lemma~4.1]{Caravenna2008}
\begin{equation}\label{eq-hitting-prob}
    \mathbf{P}_{x} \left( \exists n \geq 0, S^{+}_n < y \right) = 1- \frac{R(x-y)}{R(x)}. 
\end{equation}
In particular, by \eqref{eq-renewal-thm}, there exists $c>0$, such that for all  $ j \geq 0$, 
$$\sup_{x \geq j+2}  \mathbf{P}_{x} \left( \exists n \geq 0,  S_{n}^{+} < j+1 \right) \leq    1-c/(j+1). $$ Consequently,  we obtain that 
\begin{align}
    &   \sup_{x\ge0} \mathbf{E}_{x} \left[  \exp\left(  \lambda \sum_{m=1}^{\zeta(j)}  \mathcal{Q}_{m}(j) \right)   \right]
     \leq  1+  \sum_{m=1}^{\infty}   E_{\eqref{eq-exp-moments-m-excursion}}(\lambda, m;j) \\
      &  \leq 1+ e^{C_o \lambda  }  \sum_{m=1} \left( e^{C_o \lambda  }    \big(1-\frac c{j+1}\big)   \right) ^{m-1} = 1+ \frac{e^{C_o \lambda}}{1- e^{C_o \lambda  }(1-\frac{c}{j+1}) },
\end{align}
as long as $e^{C_o \lambda  }    \big(1-\frac c{j+1}\big) < 1$. Choosing  $\lambda=\lambda_{j}:= \frac{\delta}{C_o (1+j)^2}$, for $\delta>0$ sufficiently small, we have $e^{C_o \lambda_j  }    \big(1-\frac c{j+1}\big) < 1$ for all $j\ge 1$ and   $\left( 1- e^{ \frac{\delta}{(1+j)^2} }    \big(1-\frac c{j+1}\big)  \right)^{-1}=  O(j)  $ as $j \to \infty$. Together with \eqref{eq-bound-27}, it  follows that there exists a  constant $C'>0$ such that 
\begin{equation}
    \sup_{x\ge0} \mathbf{E}_x \left[  \exp\left(   \lambda_{j} \sum_{k \geq 0} \ind{S_k^+ \in [j,j+1)} Q_{k+1}   \right)   \right] \leq C'(1+ j)  \text{ for any } j \geq 0.
\end{equation} 
Substituting this inequality into the previous bound \eqref{eq-bound-1} yields that for some $C''<\infty$,
\begin{align}
    \mathbf{P}_x \left(  \sum_{k=0}^{\infty} R( S^{+}_{k} )e^{- S^{+} _{k} }Q_{k+1} > M \right) 
    \leq  \sum_{j \geq 0}^{\infty} C'(1+j) e^{  - c \lambda_{j} e^{j/3 }M } \le C''e^{- M/C''}.
\end{align}
This shows that there exists $\epsilon'>0$, such that
\[
   \sup_{x \ge 0} \mathbf{E}_{x} \left[ \exp \left(   \epsilon'  \sum_{k=0}^{\infty} R( S^{+}_{k} )e^{- S^{+} _{k} } Q_{k+1}\right) \right] \eqqcolon K <\infty.
\]
Letting $\epsilon\le \epsilon'$ such that $K^{\epsilon/\epsilon'} \le 2$, we obtain the statement of the theorem by an application of Jensen's inequality.

It remains to  prove \eqref{eq-exp-moments-per-excursion}. Let us denote $ L(j) :=  \sum_{k=   \tau^{\mathrm{sta}}_{1}(j) } ^{\tau^{\mathrm{end}}_{1}(j)-1} \ind{ S^+_k \in [j,j+1) }  $  the  local time in $[j,j+1)$  during the first excursion.   
 Furthermore, for each $k \geq 1$, let $\sigma^{(k)}_{j}$ denote the time of the $k$-th visit of $(S^+_n)$ to the interval $[j,j+1)$, prior to its first hitting of $[j+2,\infty)$.  Set $\sigma^{(k)}_{j}=\infty$ if no such visit occurs.
 According to the definition we have 
 \begin{equation}\label{eq-Q-1-j}
  \mathcal{Q}_{1}(j):=  \sum_{k=   \tau^{\mathrm{sta}}_{1}(j) }^{\tau^{\mathrm{end}}_{1}(j)-1  } \ind{S_k^+ \in [j,j+1)} Q_{k+1}  = \sum_{k=1}^{ L(j)} Q_{ \sigma^{(k)}_{j} +1 }
 \end{equation}   
Thus it follows that 
\begin{equation}
    F_{\eqref{eq-exp-moments-per-excursion}}(\lambda;j)    
    \le \sum_{m=1}^\infty \sup_{y \in [j,j+1)} \mathbf{E}_{y} \left[  \exp\left(  \lambda \sum_{k=1}^{m}  Q_{ \sigma^{(k)}_{j} +1 }    \right)   \ind{ L(j)  = m } \right] . 
\end{equation} 
By applying the Cauchy-Schwarz inequality,  we obtain 
\begin{align}
    F_{\eqref{eq-exp-moments-per-excursion}}(\lambda;j)  \leq \sum_{m=1}^\infty    I_{m}   ^{1/2} \, \times  \sup_{y \in [j,j+1)} \mathbf{P}_{y} \left( L(j)  = m   \right)^{1/2}  , 
\end{align} 
where 
\begin{equation}
  I_{m} \coloneqq  \sup_{y \in [j,j+1)} \mathbf{E}_{y} \left[  \exp \left( 2\lambda  \sum_{k=1}^{ m } Q_{ \sigma^{(k)}_{j} +1 }   \right) \ind{\sigma^{(m)}_{j} < \infty}   \right] .
\end{equation} 
Note that $(\sigma^{(k)}_{j})_{k \ge 1}$ are   stopping times with respect to the natural filtration of $(S^+_{n})$. 
Moreover, since the time index is discrete, our assumption \eqref{AssumptionConditionallLaw} automatically applies to these stopping times. By conditioning on $\{(S^+_{n},Q_{n}) : n \leq  \sigma^{(m)}_{j}\}$, it follows that 
\begin{align}
  I_{m}   
 & \leq  I_{m-1} \times  \sup_{x \in [j,j+1)}  \mathbf{E} \left[  \exp \left(  2\lambda Q_{\sigma^{(m)}_{j} +1 }   \right)\  \Big|\ \sigma^{(m)}_{j} < \infty,  S^+_{ \sigma^{(m)}_{j} }= x   \right]    \\
& =I_{m-1}\,  \sup_{x \in [j,j+1)}  \mathbf{E}_x \left[ e^{ 2\lambda Q_{1 } }     \right]  .  
\end{align}
Let $\delta$ and $C_\delta$ be as in \eqref{eq-as-Q-moment}.
Iterating the last inequality and using Jensen's inequality, we get for $\lambda \leq \delta/2$,
\begin{equation} 
 I_{m}    \leq ( \sup_{y\ge0}  \mathbf{E}_{y} \left[  e^{ \delta  Q_{1} }    \right] )^{\frac{2 \lambda}{\delta} m } \leq e^{\lambda C_{\delta}'m },
\end{equation} 
where   $C_{\delta}'= 2C_{\delta}/\delta$.

 Next  we show that $L(j)$ is stochastically dominated by a geometric random variable with success probability $c_1>0$ where $c_1$ is a constant independent of $j$. To this end, we begin by selecting two constants $\ell$ and $c_2$, both independent of $j$, such that 
\begin{equation}
    \inf_{y \in [j,j+1)} \mathbf{P}_{y} \left(  S^{+}_{\ell} > j+2 \right) \geq c_2 > 0 .
\end{equation}
Observe that the event $\{L(j)\ge n\ell\}=\{\sigma_j^{(n\ell)}<\infty\}$ implies two things. First, $\sigma_j^{(\ell)}<\infty$, so during its first $\ell$ steps the chain never exceeds level $j+2$. Second, by the strong Markov property, if we restart the chain at time $\sigma_j^{(\ell)}$, it must accumulate at least $(n-1)\ell$ visits to $[j,j+1)$ before hitting $[j+2,\infty)$.
In summation, we have 
\begin{align}
   \sup_{y \in [j,j+1)} \mathbf{P}_{y} (  L(j) > n \ell) 
   & \leq (1-c_2) \sup_{y \in [j,j+1)} \mathbf{P}_{y} ( L(j)> (n-1) \ell ) \\
   & \leq (1-c_2)^{n} . 
\end{align}
Hence there exists $C_{3}<\infty$ independent of $j$ such that $\sup_{y \in [j,j+1)} \mathbf{P}_{y} \left( L(j)  =m \right)\le C_3 e^{- c_{3}m} $.
 
Combining the previous inequalities we finally get,
\begin{equation}
    F_{\eqref{eq-exp-moments-per-excursion}}(\lambda;j) \le C_3 \sum_{m=1}  e^{ \lambda C'_{\delta} m/ 2}   e^{- c_{3}m/2} < \infty,
\end{equation} 
provided that $\lambda < \frac{c_{3}}{C_{\delta}'}$. Equation~\eqref{eq-exp-moments-per-excursion} then follows by another application of Jensen's inequality. This concludes the proof.
\end{proof}
  
\begin{remark}\label{rmk-spine-argu-to-W}
We now explain how to use the spine decomposition technique from the proof of Theorem~\ref{thm-tail-D-x} in order to give an alternative proof of Theorem~\ref{thm-tail-W-x}, valid under the assumptions stated in Theorem~\ref{thm-tail-W-x}. We aim to use induction to demonstrate that, for all $\theta \le \bar{\theta}$ and $x \ge 0$, the following bound holds: 
\begin{equation} 
  \psi_{n}'( \theta,x) = \E \left[ W^{(x)}_{n}  e^{ \theta W^{(x)}_{n}} \right]  = e^{-x} \hat{\E}^{*} [ e^{\theta  W^{(x)}_{n} } ] \leq 2e^{-x},   
\end{equation} 
where $\hat{\P}^{*}$ is the law of BRW with a spine defined via the change of measure $ \dif \hat{\P}^{*} |_{\mathcal{F}_{n}} = W_{n}  \dif \P |_{\mathcal{F}_{n}}$.

As in the proof  of Lemma \ref{lem-exp-mon-D-x}, define  $B^{(x)}_{n}:= \sum_{k=0}^{n} e^{- [x+V(w_{k})]} \ind{x+\infV(w_k) \ge 0} Q_{k+1}$ 
where $Q_{k+1}:=   \sum_{u \in \mathcal{B}(w_{k+1})} e^{- \Delta V(u) }$. 
It suffices to show that $\sup_{x}\sup_{n}\E^{*}[e^{\theta B^{(x)}_{n} }]<\infty $. The process  $\{ ( V(w_{n}) )_{n}, \hat{\P}^{*} \} $ has the same law as $\{(S_{n})_{n \ge 0}, \mathbf{P}_x\} $ (see e.g. \cite{Zhan15}) with $\mathbf{E}[S_{n}]=-\Phi'(1)>0$.  Therefore it suffices to verify an analogous result to Lemma~\ref{lem-CondRW}: for $\epsilon>0$ sufficiently small, we have
\begin{equation}
  \sup_{x \ge 0} \mathbf{E}_{x} \left[ \exp \left(   \epsilon  \sum_{k=0}^{\infty} e^{- S_{k} } \ind{S_{k} \ge 0}  Q_{k+1}\right) \right]  \le 2.
\end{equation}
This is simpler than Lemma \ref{lem-CondRW}, due to the positive drift of the random walk $(S_{n})_{n \ge 0}$: Equation \eqref{eq-hitting-prob} can be simply replaced by $
\mathbf{P}_{x}(\exists n \geq 0: S_n<x) <1$ and the rest of the proof remains identical. 
\end{remark}


\section*{Acknowledgement}

This material is based upon work supported by the National Science
Foundation under Grant No. DMS-1928930, while the authors were in
residence at the Simons Laufer Mathematical Sciences Institute in
Berkeley, California, during the Spring 2025 semester.
Heng Ma further acknowledges partial support from grants National Key R \& D program of China (No. 2023YFA1010103) and NSFC Key Program (Project No. 12231002). 
Pascal Maillard further acknowledges partial support from Institut Universitaire de France, the MITI interdisciplinary program 80PRIME GEx-MBB and the ANR MBAP-P (ANR-24-CE40-1833) project.

 
 \bibliographystyle{alpha}
 \bibliography{biblio}

\begin{thebibliography}{BBHM17}

\bibitem[AIR09]{AIR09}
Gerold Alsmeyer, Alex Iksanov, and Uwe R{\"o}sler.
\newblock On distributional properties of perpetuities.
\newblock {\em J. Theor. Probab.}, 22(3):666--682, 2009.

\bibitem[BBHM17]{BerestyckiBrunetHarrisMilosBranchingBrownian2017}
Julien Berestycki, {\'E}ric Brunet, Simon~C. Harris, and Piotr Mi{\l}o{\'s}.
\newblock Branching {{Brownian}} motion with absorption and the all-time minimum of branching {{Brownian}} motion with drift.
\newblock {\em Journal of Functional Analysis}, 273(6):2107--2143, September 2017.

\bibitem[Big77]{Biggins77}
J.~D. Biggins.
\newblock Martingale convergence in the branching random walk.
\newblock {\em J. Appl. Probab.}, 14:25--37, 1977.

\bibitem[Big98]{Biggins98}
J.D. Biggins.
\newblock Lindley-type equations in the branching random walk.
\newblock {\em Stochastic Processes and their Applications}, 75(1):105--133, 1998.

\bibitem[BIM21]{BIM21}
Dariusz Buraczewski, Alexander Iksanov, and Bastien Mallein.
\newblock On the derivative martingale in a branching random walk.
\newblock {\em Ann. Probab.}, 49(3):1164--1204, 2021.

\bibitem[BK04]{BK04}
J.~D. Biggins and A.~E. Kyprianou.
\newblock Measure change in multitype branching.
\newblock {\em Adv. Appl. Probab.}, 36(2):544--581, 2004.

\bibitem[BK05]{BK05}
J.~D. Biggins and A.~E. Kyprianou.
\newblock Fixed points of the smoothing transform: the boundary case.
\newblock {\em Electron. J. Probab.}, 10:609--631, 2005.
\newblock Id/No 17.

\bibitem[Bur09]{Buraczewski09}
Dariusz Buraczewski.
\newblock On tails of fixed points of the smoothing transform in the boundary case.
\newblock {\em Stochastic Processes Appl.}, 119(11):3955--3961, 2009.

\bibitem[CC08]{Caravenna2008}
Francesco Caravenna and Lo{\"i}c Chaumont.
\newblock Invariance principles for random walks conditioned to stay positive.
\newblock {\em Annales de l'institut Henri Poincare (B) Probability and Statistics}, 44(1):170--190, 2008.

\bibitem[CdRM24]{CDM24}
Xinxin Chen, Loïc de~Raphélis, and Heng Ma.
\newblock Branching random walk conditioned on large martingale limit, 2024.
\newblock arXiv:2408.05538.

\bibitem[Che15]{Chen15}
Xinxin Chen.
\newblock {A necessary and sufficient condition for the nontrivial limit of the derivative martingale in a branching random walk}.
\newblock {\em Advances in Applied Probability}, 47(3):741 -- 760, 2015.

\bibitem[Don80]{Doney80}
R.~A. Doney.
\newblock Moments of ladder heights in random walks.
\newblock {\em Journal of Applied Probability}, 17(1):248–252, 1980.

\bibitem[DSS23]{DalySchulteShneerFirstPassage2023}
Fraser Daly, Matthias Schulte, and Seva Shneer.
\newblock First passage percolation on {{Erd{\H o}s-R{\'e}nyi}} graphs with general weights, August 2023.

\bibitem[Fel71]{Feller1971}
William Feller.
\newblock {\em An Introduction to Probability Theory and Its Applications. {{Vol II}}.}
\newblock John Wiley {\textbackslash}\& Sons Inc., New York, 1971.

\bibitem[Gui90]{Guivarc'h90}
Yves Guivarc'h.
\newblock Sur une extension de la notion de loi semi-stable. ({On} an extension of the notion of semi-stable law).
\newblock {\em Ann. Inst. Henri Poincar{\'e}, Probab. Stat.}, 26(2):261--285, 1990.

\bibitem[Iks16]{IksanovRenewalTheory2016}
Alexander Iksanov.
\newblock {\em Renewal {{Theory}} for {{Perturbed Random Walks}} and {{Similar Processes}}}.
\newblock Probability and {{Its Applications}}. Springer International Publishing, Cham, 2016.

\bibitem[Liu00]{Liu00}
Quansheng Liu.
\newblock On generalized multiplicative cascades.
\newblock {\em Stochastic Processes and their Applications}, 86(2):263--286, 2000.

\bibitem[Mad16]{Madaule16b}
Thomas Madaule.
\newblock The tail distribution of the {Derivative} martingale and the global minimum of the branching random walk.
\newblock Preprint, {arXiv}:1606.03211 [math.{PR}] (2016), 2016.

\bibitem[Rog64]{RogozinDistributionFirst1964}
B.~A. Rogozin.
\newblock On the {{Distribution}} of the {{First Jump}}.
\newblock {\em Theory of Probability \& Its Applications}, 9(3):450--465, January 1964.

\bibitem[Rö92]{Rosler95}
Uwe Rösler.
\newblock A fixed point theorem for distributions.
\newblock {\em Stochastic Processes and their Applications}, 42(2):195--214, 1992.

\bibitem[Shi15]{Zhan15}
Zhan Shi.
\newblock {\em Branching random walks. {\'E}cole d'{\'E}t{\'e} de {Probabilit{\'e}s} de {Saint}-{Flour} {XLII} -- 2012}, volume 2151 of {\em Lect. Notes Math.}
\newblock Cham: Springer, 2015.

\bibitem[Tan89]{Tan1989}
Hiroshi Tanaka.
\newblock Time reversal of random walks in one dimension.
\newblock {\em Tokyo J. Math.}, 12(1):159--174, 1989.

\end{thebibliography}

 \end{document}